\documentclass[11pt, reqno]{amsart}

\usepackage{stmaryrd}
\usepackage{amsmath}
\usepackage{amsfonts}
\usepackage{amssymb}
\usepackage[all,cmtip]{xy}           
\usepackage{xspace}
\usepackage{bbding}
\usepackage{txfonts}
\usepackage[shortlabels]{enumitem}
\usepackage{cite}

\usepackage[utf8]{inputenc}
\usepackage{tikz}
\usepackage{titletoc}

\usepackage{ifpdf}
\ifpdf
 \usepackage[colorlinks,final,backref=page,hyperindex]{hyperref}
\else
 \usepackage[colorlinks,final,backref=page,hyperindex,hypertex]{hyperref}
\fi
\usepackage[active]{srcltx}

\makeatletter

\newcommand {\emptycomment}[1]{} 

\newcommand{\nc}{\newcommand}
\newcommand{\delete}[1]{}
\nc{\mfootnote}[1]{\footnote{#1}} 
\nc{\todo}[1]{\tred{To do:} #1}

\nc{\mlabel}[1]{\label{#1}}  
\nc{\mcite}[1]{\cite{#1}}  
\nc{\mref}[1]{\ref{#1}}  
\nc{\meqref}[1]{\ref{#1}} 
\nc{\mbibitem}[1]{\bibitem{#1}} 

 \delete{
\nc{\mlabel}[1]{\label{#1}  
{\hfill \hspace{1cm}{\bf{{\ }\hfill(#1)}}}}
\nc{\mcite}[1]{\cite{#1}{{\bf{{\ }(#1)}}}}  
\nc{\mref}[1]{\ref{#1}{{\bf{{\ }(#1)}}}}  
\nc{\meqref}[1]{\eqref{#1}{{\bf{{\ }(#1)}}}} 
\nc{\mbibitem}[1]{\bibitem[\bf #1]{#1}} 
 }

\def\b{\beta}\def\btl{\blacktriangleright}\def\btr{\blacktriangleleft}
\def\c{\cdot}\def\ci{\circ}
\def\d{\delta}\def\dd{\diamondsuit}\def\D{\Delta}\def\dr{\lessdot}\def\dl{\gtrdot}

\def\l{\lambda}\def\rc{\prec}\def\lc{\succ}\def\lr{\longrightarrow}

\def\o{\otimes}\def\om{\omega}

\def\s{\sigma}\def\ti{\times}\def\tl{\triangleright}\def\tr{\triangleleft}

\nc{\mrm}[1]{{\rm #1}}\nc{\id}{\mrm{id}}\nc{\End}{\mrm{End}} \nc{\frakB}{{\mathfrak B}}
\def\1{_{(1)}}\def\2{_{(2)}}\def\3{^{-1}}
\nc{\rdd}{{r_{\cdot}}^*}   \nc{\ldd}{{\ell_{\cdot}}^*}
\nc{\rcc}{{r_{\circ}}^*}   \nc{\lcc}{{\ell_{\circ}}^*}
\nc{\rab}{{r_{A^*}}^*}   \nc{\lab}{{\ell_{A^*}}^*}
\nc{\raa}{{r_{A}}^*}   \nc{\laa}{{\ell_{A}}^*}
\nc{\rb}{{r_{A^*}}}   \nc{\lb}{{\ell_{A^*}}}
\nc{\ra}{{r_{A}}}   \nc{\la}{{\ell_{A}}}
\nc{\bu}{{\bullet}}
\nc{\ts}[1]{\textcolor{red}{Tianshui:#1}}

\newtheorem{thm}{Theorem}[section]
\newtheorem{lem}[thm]{Lemma}
\newtheorem{cor}[thm]{Corollary}
\newtheorem{pro}[thm]{Proposition}
\theoremstyle{definition}
\newtheorem{defi}[thm]{Definition}
\newtheorem{ex}[thm]{Example}
\newtheorem{rmk}[thm]{Remark}

\begin{document}

\title[]{\bf Curved $\mathcal{O}$-operator systems}

\author{Tianshui Ma}
\address{School of Mathematics and Information Science, Henan Normal University, Xinxiang 453007, China.}
\email{matianshui@htu.edu.cn}

\author{Abdenacer Makhlouf}
\address{Universit{\'e} de Haute Alsace, IRIMAS-d\'epartement  de Math{\'e}matiques,  18, rue des Fr{\`e}res Lumi{\`e}re F-68093 Mulhouse, France.}
\email{abdenacer.makhlouf@uha.fr}

\author{Sergei Silvestrov}
\address{Division of Mathematics and Physics, School of Education,
 Culture and Communication, M\"alardalen University, 72123
  V\"aster{\aa}s,
 Sweden.}
\email{sergei.silvestrov@mdu.se}

\date{\today}

\begin{abstract}
 In this paper, we introduce the notion of curved $\mathcal{O}$-operator systems as a generalization of T. Brzezi\'{n}ski's (curved) Rota-Baxter systems, and then investigate their relations with $\mathcal{O}$-operator systems, (tri)dendriform systems, pre-Lie algebras, associative Yang-Baxter pairs and quasitriangular covariant bialgebras.
 \end{abstract}

\subjclass{16T05, 16T25,16W99}

\keywords{$\mathcal{O}$-operator system; (tri)dendriform system; associative Yang-Baxter pair.}

\maketitle


\allowdisplaybreaks

\section{Introduction} \hskip\parindent
Rota-Baxter operators appeared first in probability theory in \cite{Ba}, while Rota-Baxter algebras were introduced  in the context of differential operators on commutative Banach algebras in   \cite{Ro1} and since then  intensively studied  in a wide range of areas in pure and applied mathematics, as well as in mathematical physics.  Likewise,  free Rota-Baxter algebras were discussed in \cite{Guo2,GK1}, differential Rota-Baxter algebras in \cite{GK2}, for connections with renormalization of quantum field theory  and Hopf algebra, see \cite{CK,ML}. One can refer to the book \cite{Guo1} for the detailed theory of Rota-Baxter algebras.  Dendriform algebras were introduced by Loday in  \cite{Lo}. The motivation to introduce these algebraic structures with two generating operations comes from $K$-theory. It turned out later that they are connected to several areas in mathematics and physics, including Hopf algebras, homotopy Gerstenhaber algebra, operads, homology, combinatorics and quantum field theory where they occur in the theory of renormalization of Connes and Kreimer. Later the notion of tridendriform algebra were introduced by Loday and Ronco in their study of polytopes and Koszul duality, see \cite{LR}. The relation between Rota-Baxter algebras and (tri)dendriform algebras was first observed by Aguiar \cite{Ag0}, then  discussed by Ebrahimi-Fard in \cite{EF}, see also \cite{EG,EFM1,EFM2,AB}. 

The notion of $\mathcal{O}$-operator on an associated algebra with a representation was introduced in \cite{BGN1}. Such a structure appeared independently in \cite{Uc} under the name of generalized Rota-Baxter operator.
It is an analogue of the $\mathcal{O}$-operator defined by Kupershmidt as a natural generalization of the operator form of the classical Yang-Baxter equation (\cite{Ku}). The $\mathcal{O}$-operators have also a close relationship with the associative Yang-Baxter equation.
 Later, Bai, Guo and Ni introduced the extended $\mathcal{O}$-operator generalizing the concept of $\mathcal{O}$-operators and studied their  relationships with associative Yang-Baxter equations \cite{BGN1}. $\mathcal{O}$-operators are also often called relative Rota-Baxter operators.

  In 2016, in an attempt to develop and extend aforementioned connections between Rota-Baxter algebras, dendriform algebras and infinitesimal bialgebras, T. Brzezi\'{n}ski introduced the notion of a Rota-Baxter system in \cite{Br1}. In particular it has been shown that to any Rota-Baxter system one can associate a dendriform algebra and, in fact, any dendriform algebra of a particular kind arises from a Rota-Baxter system. In \cite{Br2}, T. Brzezi\'{n}ski presented the curved version of Rota-Baxter system and investigated the relations with weak pseudotwistors, differential graded algebras and pre-Lie algebras. But the relations between curved Rota-Baxter systems and (tri)dendriform algebras was never considered. Curved Rota-Baxter systems generalize Rota-Baxter operators and algebras at least in a threefold way. First, when the curvature vanishes, the triple $(A, R, S)$ is a Rota-Baxter system and hence the choice of $S$ to be $R+\lambda id$ with $\lambda \in K$ to be $S+\lambda id$ makes it into a Rota-Baxter algebra of weight $\lambda$. On the other hand, a Rota-Baxter algebra of weight $\lambda$ is obtained from $(A, R, S, \omega)$ by setting $R=S$ and $\omega(a\otimes b)=\lambda(ab)$, for all $a, b\in A$.
Examples  of Rota-Baxter systems are obtained from quasitriangular covariant bialgebras hereby introduced as a natural extension of infinitesimal bialgebras introduced in \cite{JR} and studied in \cite{Ag1,Ag2}. In \cite{Bai1}, the author presented a notion of antisymmetric infinitesimal bialgebras and also provided a equivalent characterizations by double construction and matched pairs.

The main purpose of this paper is to  introduce and study  a notion of (curved) $\mathcal{O}$-operator system, which generalizes both T. Brzezi\'{n}ski's Rota-Baxter systems (see \cite[Definition 2.1]{Br1}) and $\mathcal{O}$-operators (see \cite[Definition 2.7]{BGN1}) and  the notion of (tri)dendriform systems. We  investigate relationships within $\mathcal{O}$-operator systems, pre-Lie algebras, (tri)dendriform systems, associative Yang-Baxter pairs and quasitriangular covariant bialgebras.
 The paper is organized as follows. In Section 2, we recall some basics and relevant  definitions. In Section 3, we introduce the notions of curved $\mathcal{O}$-operator system and (tri)dendriform systems and discuss their connections and also with associative algebras  and extended bimodules algebras. In Section 4, we discuss $\mathcal{O}$-operator systems and associative Yang-Baxter pairs. Section 5 is devoted to a class of special curved $\mathcal{O}$-operator systems, mainly  generalized Rota-Baxter algebras with respect to associative compatible pairs and  double curved Rota-Baxter systems.

 \section{Preliminaries}

 Throughout this paper, $K$ will be a field, and all vector spaces, tensor products, and
 homomorphisms are over $K$. We denote by $\id_M$ the identity map on $M$.  

 We recall first some basics and useful definitions, see \cite{Bai1,Br1,Ra12}.

\begin{defi}\mlabel{de:2.1} A {\bf Rota-Baxter algebra of weight $\l$} is an associative algebra $A$ together with a linear map $R : A \lr A$, such that
for all $a, b\in A$ and $\l\in K$,
\begin{equation}\mlabel{eq:2.1}
 R(a)R(b)=R(aR(b))+R(R(a)b)+\l ab.
 \end{equation}
 We refer to such a  Rota-Baxter algebra by a pair  $(A, R)$ and the  linear operator $R$ is called a {\bf Rota-Baxter operator of weight $\l$} on $A$.
\end{defi}

\begin{defi}\mlabel{de:2.2} Let $A$ be an associative algebra and $M$ be a linear space. Let $\tl: A\o M\lr M$ and $\tr: M\o A\lr M$ be two linear maps. Then $(M, \tl, \tr)$ is called an {\bf $A$-bimodule} if, for all  $a, b\in A$ and $x\in M$,
\begin{equation}\mlabel{eq:2.2}
 a\tl (b\tl x)=(ab)\tl x,\quad (x\tr a)\tr b=x\tr (ab), \quad (a\tl x)\tr b=a\tl (x\tr b). \end{equation}

An {\bf $A$-bimodule map $f$ from $(M, \tl_M, {_M\tr})$ to $(N, \tl_N, {_N\tr})$} is a linear map $f: M\lr N$ such that $f$ is a left $A$-module map from $(M, \tl_M)$ to $(N, \tl_N)$ and at the same time $f$ is a right $A$-module map $(M, {_M\tr})$ to $(N, {_N\tr})$.

An $A$-bimodule $(M, \tl, \tr)$ is {\bf faithful} if $a\tl M=0$ implies $a=0$ and simultaneously $M\tr a=0$ implies $a=0$.
\end{defi}

 \begin{defi}\mlabel{de:2.3} A {\bf (left) pre-Lie algebra (or left symmetric algebra)} $A$ is a vector space with a binary operation $(x, y)\mapsto xy$ satisfying,
 for all $x, y, z \in A$,
 \begin{equation}\mlabel{eq:2.3}
 (xy)z-x(yz)=(yx)z-y(xz).
 \end{equation}

 \end{defi}

 We consider the following associative analogue of the classical Yang-Baxter equation.

 \begin{defi}\mlabel{de:2.4} Let $A$ be an associative algebra. An {\bf associative Yang-Baxter pair} is a pair of elements $r, s\in A\o A$ that satisfy the following equations:
 \begin{equation}\mlabel{eq:2.4}
 r_{13}r_{12}-r_{12}r_{23}+s_{23}r_{13}=0,
 \end{equation}
 \begin{equation}\mlabel{eq:2.5}
 s_{13}r_{12}-s_{12}s_{23}+s_{23}s_{13}=0,
 \end{equation}
 that is,
 $$
 r^1\bar{r}^1\o \bar{r}^2\o r^2-r^1\o r^2\bar{r}^1\o \bar{r}^2+r^1\o s^1\o s^2r^2=0,
 $$
 $$
 s^1r^1\o r^2\o s^2-s^1\o s^2\bar{s}^1\o \bar{s}^2+s^1\o \bar{s}^1\o \bar{s}^2s^2=0,
 $$
 where $\bar{r}$ and $\bar{s}$ are the copies of $r$ and $s$, respectively.

\end{defi}

We note that when $s=r$ in Definition \mref{de:2.4}, then one can get the associative Yang-Baxter equation in \cite{Ag1,Ag2}.

\section{Curved $\mathcal{O}$-operator systems}\mlabel{se:cos} In this section, we introduce a notion of curved $\mathcal{O}$-operator system,  which generalizes both T. Brzezi\'{n}ski's Rota-Baxter system (see \cite[Definition 2.1]{Br1}) and the $\mathcal{O}$-operator (see \cite[Definition 2.7]{BGN1}).

\subsection{Definition and Examples}
 \begin{defi}\mlabel{de:2.1-1} A system $(A, M, R, S, \omega)$ consisting of an associative algebra $A$, an $A$-bimodule $(M, \triangleright, \triangleleft)$ and three linear maps $R, S: M\longrightarrow A$, $\omega: M\otimes M\longrightarrow M$ is called a {\bf curved $\mathcal{O}$-operator system associated to $(M, \triangleright, \triangleleft)$} if, for all $x, y\in M$,
 \begin{align}\mlabel{eq:2.1-1}
 R(x)R(y)&=R(R(x)\triangleright y+x\triangleleft S(y)+\omega(x\otimes y))=R(R(x)\triangleright y+x\triangleleft S(y))+R\omega(x\otimes y),
\\
\mlabel{eq:2.2-1}
 S(x)S(y)&=S(R(x)\triangleright y+x\triangleleft S(y)+\omega(x\otimes y))=S(R(x)\triangleright y+x\triangleleft S(y))+S\omega(x\otimes y).
 \end{align}

When $\omega=0$, we call $(A, M, R, S)$ an {\bf $\mathcal{O}$-operator system associated to $(M, \triangleright, \triangleleft)$}.

A {\bf morphism of (curved) $\mathcal{O}$-operator systems from $(A, M, R, S, \omega)$ to $(B, N, P, T, \nu)$} is a pair $(f, g)$, where $f: A\longrightarrow B$ is an algebra map, and $g: M\longrightarrow N$ is an $A$-bimodule map such that $f\circ R=P\circ g$, $f\circ R=T\circ g$, $f\circ S=P\circ g$, $f\circ S=T\circ g$ (and $g\circ \omega=\nu\circ (g\otimes g)$).

In the sequel, we often denote $\om(x\o y)$ by $x\circ y$ and $(A, M, R, S, \omega)$ by $(A, M, R, S, \circ)$.
\end{defi}

Next we provide some examples of (curved) $\mathcal{O}$-operator system.
\begin{ex}\mlabel{ex:3.2}\mlabel{ex:2.2-1}
\begin{enumerate}[label=\textup{(\alph*)},ref=\textup{\alph*},leftmargin=*]
\item \mlabel{it:3.2-1}\mlabel{it:2.2-1-2}
Let $(A, \cdot)$ be an associative algebra and $(A, \cdot_\ell, \cdot_r)$ be  an $A$-bimodule, where $\cdot_\ell, \cdot_r$ denote the left and right multiplication, respectively. Then an (curved) $\mathcal{O}$-operator system $(A, A, R, S)$ ($(A, A, R, S, \omega)$) associated to $(A, \cdot_\ell, \cdot_r)$ is exactly a (curved) Rota-Baxter system in \cite[Definition 2.1]{Br1} (\cite[Definition 1.1]{Br2}).
\item \mlabel{it:3.2-2} If $R=S$, then we  get the usual $\mathcal{O}$-operator for $(M, \tl, \tr)$ studied in \cite{BGN1}. \mlabel{it:2.2-1-3} When $R=S$ and $\omega(x\otimes y)=\lambda xy$ in (\mref{it:2.2-1-2}), we  get that $R$ is a Rota-Baxter operator of weight $\lambda$ in \cite{Ba,Ro1}. 
\item \mlabel{it:2.2-1-4} If $(M, \omega, \triangleright, \triangleleft)$ is an $A$-bimodule algebra and $R=S$, then a  curved $\mathcal{O}$-operator system $(A, M, R$, $R, \lambda \omega)$ is the $\mathcal{O}$-operator of weight $\lambda$ associated to $(M, \omega, \triangleright, \triangleleft)$ (see \cite[Definition 2.7]{BGN1}).
\item \mlabel{it:3.2-4}  $(A, M, R, 0)$ is an $\mathcal{O}$-operator system associated to $(M, \tl, \tr)$ if and only if
 \begin{equation}\mlabel{eq:3.3}
 R(x)R(y)=R(R(x)\tl y).
\end{equation}
 $(A, M, 0, R)$ is an $\mathcal{O}$-operator system associated to $(M, \tl, \tr)$ if and only if
 \begin{equation}\mlabel{eq:3.4}
 R(x)R(y)=R(x\tl R(y)).
 \end{equation}
 \item \mlabel{it:2.2-1-5} Let $(M, \triangleright, \triangleleft)=(A, \c_\ell, \c_r)$, $R=S$ and $\omega_1=\omega_2=-\c \circ (R\otimes R)$. Then a curved $\mathcal{O}$-operator system $(A, A, R, R, \omega)$ is called Reynolds algebra in \cite{Uc2}.
 \item \mlabel{it:2.2-1-6} Let $A$ be an associative algebra with unit $1_A$, $R=S$ and $\omega(a\otimes b)=-aR(1_A)b$. Then a curved $\mathcal{O}$-operator system $(A, A, R, R, \omega)$ is the TD-algebra in \cite{Le}.
 \item \mlabel{it:2.2-1-7} Let $(M, \triangleright, \triangleleft)=(A, \c_\ell, \c_r)$, $R=S$ and $\omega(x\otimes y)=-R(x\c y)$. Then $R$ is a Nijenhuis operator in \cite{CGM}.
 \end{enumerate}
 \end{ex}

 \begin{ex} \mlabel{it:3.2-5}  As a generalization of $\s$-twisted Rota-Baxter algebra which is motivated by Jackson $q$-integral (see \cite{Br1,EFM2}), we give the definition of $\s$-twisted $\mathcal{O}$-operator algebra.

 Let $A$ be an associative algebra, $R: M\lr A$ be a linear map, $(M, \tl, \tr)$ be an $A$-bimodule, and $\s: A\lr A$ be an algebra map. Then  $(A, M, R)$ is called a {\it $\s$-twisted $\mathcal{O}$-operator algebra} if
 \begin{equation}\mlabel{eq:3.5}
 R(x)R(y)=R(R(x)\tl y+x\tr R^{\s}(y)),
 \end{equation}
 where $x, y\in M$ and $R^{\s}=\s\ci R$.

 If $(A, M, R)$ is a $\s$-twisted $\mathcal{O}$-operator algebra, then $(A, M, R, R^{\s})$ is an $\mathcal{O}$-operator system associated to $(M, \tl, \tr)$.
 \end{ex}

\begin{ex}\mlabel{ex:2.3-1}
 Let $A$ be a 2-dimensional associative algebra where the multiplication is defined, with respect to a basis  $\{e_1,e_2\}$,  by
\begin{align*}
& e_1\cdot e_1=e_1, && e_1\cdot  e_2=  e_2,
& e_2\cdot e_1=e_2,  && e_2 \cdot e_2=
e_2.
\end{align*}
Consider the $A$-bimodule structure on a 1-dimensional vector space $M$ generated by $x$, where the right and left actions are  defined as
\begin{align*}
& e_1\triangleright x=x, &&
e_2\triangleright x=0,  &
 x\triangleleft e_1=0, &&
x\triangleleft e_2=0.
\end{align*}
The corresponding curved $\mathcal{O}$-operator systems which are not $\mathcal{O}$-operator systems  are defined  as
\begin{itemize}
\item $R(x)=0, \  S(x)=p  e_1, \ w(x,x)=p x,$
\item $R(x)=0, \  S(x)=p  e_2, \ w(x,x)=p x,$
\item $R(x)=0, \  S(x)=p(e_1-  e_2), \ w(x,x)=p x,$
\item $R(x)=p  e_2, \  S(x)=0, \ w(x,x)=p x,$
\item $R(x)=p  e_2, \  S(x)=p e_2, \ w(x,x)=p x,$
\item $R(x)=p  e_2, \  S(x)=p e_1, \ w(x,x)=p x,$
\item $R(x)=p  e_2, \  S(x)=p (e_1-e_2), \ w(x,x)=p x,$
\end{itemize}
where $p$ is a parameter. Moreover, we have the following $\mathcal{O}$-operator systems ($w(e_1,e_1)=0$).
\begin{itemize}
\item $R(x)=p  e_1, \  S(x)=0,$
\item $ R(x)=p  e_1, \   S(x)=p  e_2,$
\item $ R(x)=p  e_1, \   S(x)=p  e_1,$
\item $ R(x)=p e_1, \   S(x)=p (e_1-e_2),$
\item $ R(x)=p( e_1- e_2), \ S(x)=0,$
\item $ R(x)=p( e_1-e_2), \ S(x)=p e_2,$
\item $ R(x)=p( e_1- e_2), \ S(x)=p e_1,$
\item $ R(x)=p( e_1- e_2), \ S(x)=p( e_1- e_2).$
\end{itemize}

Now, consider the $A$-bimodule structure on the same vector space and  where the right and left actions are  defined as
\begin{align*}
& e_1\triangleright x=x, &&
e_2\triangleright x=x,  &
 x\triangleleft e_1=x, &&
x\triangleleft e_2=x.
\end{align*}
The corresponding curved $\mathcal{O}$-operator systems which are not  $\mathcal{O}$-operator systems  are defined  as
\begin{itemize}
\item $R(x)=0, \  S(x)=p  (e_1-e_2), \ w(x,x)=p x,$
\item $R(x)=p e_2, \  S(x)=p e_2, \ w(x,x)=-p x,$
\item $R(x)=p e_2, \  S(x)=p e_1, \ w(x,x)=-p x,$
\item $R(x)=p e_1, \  S(x)=p e_2, \ w(x,x)=-p x,$
\item $R(x)=p e_1, \  S(x)=p e_1, \ w(x,x)=-p x,$
\item $R(x)=p (e_1-e_2), \  S(x)=0, \ w(x,x)=p x,$
\item $R(x)=p (e_1-e_2), \  S(x)=p (e_1-e_2), \ w(x,x)=p x,$
\end{itemize}
where   $p$ is a parameter.

Finally, we consider  the $A$-bimodule structure on the 2-dimensional  vector space generated by $\{x,y\}$ and  where the right and left actions are  defined as
\begin{align*}
 e_1\triangleright x=x, && e_1\triangleright y=y, &&
e_2\triangleright x=x,  && e_2\triangleright y=y, && \\
 x\triangleleft e_1=x, && y\triangleleft e_1=y, &&
x\triangleleft e_2=x, &&y\triangleleft e_2=y.&&
\end{align*}
We have a  curved $\mathcal{O}$-operator system defined by
$$
R(x)=0, \ R(y)=p_1e_1, \  S(x)=p_2  (e_1-e_2), \ S(y)=p_1e_2,
$$
$$
w(x, x)= p_2 x, \ w(x,y)=-p_1 x, \ w(y,x)=-p_1 x,\ w(y,y)=-p_1 y,
$$
where $p_1$ and $p_2$ are parameters.
\end{ex}

\begin{pro}\mlabel{pro:3.3} Let $(A, \cdot)$ be an associative algebra and $(M, \tl, \tr)$, $(A, \cdot_\ell, \cdot_r)$ be  two $A$-bimodules. If $R, S: M\lr A$ are $A$-bimodule maps, then $(A, M, R, S)$ is an $\mathcal{O}$-operator system associated to $(M, \tl, \tr)$ if and only if, for all $x, y\in M$,
 \begin{equation}\mlabel{eq:3.6}
 R(x)S(y)=0.
 \end{equation}
 \end{pro}

 \begin{proof} When $R, S: M\lr A$ are $A$-bimodule maps, then the conditions \eqref{eq:2.1-1} and \eqref{eq:2.2-1} for $\omega=0$ are equivalent to  \eqref{eq:3.6}, respectively.
 \end{proof}

\subsection{Extended bimodule algebras}

 The notion of $A$-bimodule algebra was introduced in \cite{BGN1}, it can be seen as a special case of matched pair of two algebras in \cite{Bai1}.

 \begin{defi}\mlabel{de:2.4-1} Let $(A, \cdot)$ and $(M, \circ)$ be two associative algebras and $\triangleright: A\otimes M\longrightarrow M$, $\triangleleft: M\otimes A\longrightarrow M$, $R, S: M\longrightarrow A$ be four linear maps. We call $(M, \circ, \triangleright, \triangleleft, R, S)$ an {\bf extended $A$-bimodule algebra} if $(M, \triangleright, \triangleleft)$ is an $A$-bimodule and the following conditions hold, for all  $x, y, z\in M$,
 \begin{equation}\mlabel{eq:2.3-1}
 R(x)\triangleright (y\circ z)=(R(x)\triangleright y)\circ z,\ (x\circ y)\triangleleft S(z)=x\circ (y\triangleleft S(z)),\ x\circ (R(y)\triangleright z)=(x\triangleleft S(y))\circ z.
 \end{equation}
 \end{defi}

 \begin{rmk}\mlabel{rmk:2.5-1}
 \begin{enumerate}[label=\textup{(\alph*)},ref=\textup{\alph*},leftmargin=*]
 \item \mlabel{it:rmk:3.5-1} If the product $\circ$ on $M$ is zero, then an extended $A$-bimodule algebra is exactly an $A$-bimodule.
 \item \mlabel{it:rmk:3.5-2} If $(M, \circ, \triangleright, \triangleleft)$ is an $A$-bimodule algebra, then $(M, \circ, \triangleright, \triangleleft, R, S)$ is an extended $A$-bimodule algebra if and only if, for all $x, y, z\in M$,
 \begin{equation}\mlabel{eq:2.4-a}
 x\circ ((R-S)(y)\triangleright z)=0~~\hbox{or}~~(x\triangleleft(R-S)(y))\circ z=0,
 \end{equation}
In fact, since $(M, \circ, \triangleright, \triangleleft)$ is an $A$-bimodule algebra, it is obvious that  \eqref{eq:2.4-a} is equivalent to the last equation in \eqref{eq:2.3-1}.
 \item \mlabel{it:rmk:3.5-3} Let $(A, \cdot)$ and $(A, \circ)$ be associative algebras. Then an extended $A$-bimodule algebra $(A, \circ, \c_\ell, \c_r, \id_A, \id_A)$ is exactly a totally compatible associative dialgebra introduced in \mcite{ZBG}, which will be studied in the  last section.
 \end{enumerate}
 \end{rmk}

 \begin{thm}\mlabel{thm:2.14} Let $A$ be an associative algebra, $(M, \circ, \triangleright, \triangleleft, R, S)$ be an extended $A$-bimodule algebra and $\star: M\otimes M\longrightarrow M$ be a linear map defined by
 \begin{equation}\mlabel{eq:2.21-a}
 x\star y=R(x)\triangleright y+x\triangleleft S(y)+x\circ y,\quad x, y\in M.
 \end{equation}
 Then $(M, \star)$ is an associative algebra if and only if
 \begin{equation}\mlabel{eq:2.22-a}
 (R(x)R(y)-R(x\star y))\triangleright z=x\triangleleft (S(y)S(z)-S(y\star z)).
 \end{equation}
\end{thm}

\begin{proof} For all $x, y, z\in M$, we have
 \begin{eqnarray*}
 (x\star y)\star z
 &=&R(x\star y)\triangleright z+(R(x)\triangleright y)\triangleleft S(z)+(x\triangleleft S(y))\triangleleft S(z)+(x\circ y)\triangleleft S(z)\\
 &&+(R(x)\triangleright y)\circ z+(x\triangleleft S(y))\circ z+(x\circ y)\circ z\\
 &\stackrel{\eqref{eq:2.2}\eqref{eq:2.3-1}}{=}&R(x\star y)\triangleright z+R(x)\triangleright (y\triangleleft S(z))+x\triangleleft (S(y)S(z))+x\circ (y\triangleleft S(z))\\
 &&+R(x)\triangleright (y\circ z)+x\circ (R(y)\triangleright z)+x\circ (y\circ z)
 \\
 x\star (y\star z)
 &\stackrel{\eqref{eq:2.2}}{=}&(R(x)R(y))\triangleright z+R(x)\triangleright (y\triangleleft S(z))+R(x)\triangleright (y\circ z)+x\triangleleft S(y\star z)\\
 &&+x\circ (R(y)\triangleright z)+x\circ (y\triangleleft S(z))+x\circ (y\circ z).
 \end{eqnarray*}
 Therefore, we finish the proof.   \end{proof}

 \begin{cor}\mlabel{cor:2.15} Let $A$ be an associative algebra, $(M, \circ, \triangleright, \triangleleft, R, S)$ be an extended $A$-bimodule algebra. If $(A, M, R, S, \circ)$ is a curved $\mathcal{O}$-operator system associated to $(M, \triangleright, \triangleleft)$, then $(M, \star)$ is an associative algebra, where $\star$ is given by
\eqref{eq:2.21-a}.
 \end{cor}

 \begin{proof} By  \eqref{eq:2.1} and \eqref{eq:2.2}, we can get $R(x)R(y)=R(x\star y)$ and $S(y)S(z)=S(y\star z)$. Thus \eqref{eq:2.22-a} holds.   \end{proof}

 \begin{rmk}
 When $\circ=0$ in Corollary \mref{cor:2.15}, then we  obtain the corresponding result for $\mathcal{O}$-operator systems (not curved).
 \end{rmk}

 \begin{pro}\mlabel{pro:2.17} Let $A$ be an associative algebra, $(M, \circ, \triangleright, \triangleleft)$ be an $A$-bimodule algebra, $R: M\longrightarrow A$ a linear map. Then the operation given by
 \begin{equation}\mlabel{eq:2.23-a}
 x\star_R y=R(x)\triangleright y+x\triangleleft R(y)+x\circ y
 \end{equation}
 is associative if and only if
 \begin{equation}\mlabel{eq:2.24-a}
 (R(x)R(y)-R(x\star_R y))\triangleright z=x\triangleleft (R(y)R(z)-R(y\star_R z)),  \end{equation}
 for all  $x, y, z\in M$.
 \end{pro}

 \begin{proof} Let $R=S$ in Theorem \mref{thm:2.14}.   \end{proof}

 \begin{rmk}\mlabel{rmk:2.18}
 Proposition \mref{pro:2.17} is exactly \cite[Lemma 2.12]{BGN1} by replacing $x\circ y$ by $\lambda x\circ y$.
 \end{rmk}

\begin{defi}\mlabel{de:2.19}
 {\rm Assume  $CharK\neq 2$. We call
 \begin{equation}\mlabel{eq:2.25-a}
 \alpha=\frac{R+S}{2} \hbox{~~and~~}\beta=\frac{R-S}{2}
 \end{equation}
 the {\bf symmetrizer and antisymmetrizer of $R$ and $S$}, respectively (see [Sec. 2.3]\cite{BGN1}). }
\end{defi}

 \begin{defi} \cite[Definition 2.7]{BGN1} \mlabel{de:cor:2.20}
 Let $A$ be an associative algebra, $(M, \triangleright, \triangleleft)$ be an $A$-bimodule, $k\in K$. A linear map $\b: M\lr A$ is called {\bf a balanced linear map of mass $k$} if, for all $x, y\in M$,
 \begin{equation}\mlabel{eq:2.26-a}
 k \beta(x)\triangleright y=k x\triangleleft \beta(y).
 \end{equation}
 \end{defi}

\begin{cor}\mlabel{cor:2.20}
 Let $A$ be an associative algebra, $(M, \circ, \triangleright, \triangleleft)$ be an $A$-bimodule algebra, $R, S: M\longrightarrow A$ two linear maps and $\alpha, \beta$ their symmetrizer and antisymmetrizer defined by  \eqref{eq:2.25-a}. If $\beta$ is a balanced linear map of mass 1, then $(M, \star)$, where $\star$ is given by  \eqref{eq:2.21-a}, is an associative algebra if and only if $\alpha$ satisfies  \eqref{eq:2.24-a}.
\end{cor}

 \begin{proof} Since $\beta$ is balanced, for all $x, y\in M$,
 $$
 x\star y=R(x)\triangleright y+x\triangleleft S(y)+x\circ y=\alpha(x)\triangleright y+x\triangleleft \alpha(y)+x\circ y=x\star_\alpha y.
 $$
 Then the conclusion follows from Proposition \mref{pro:2.17}.   \end{proof}

\begin{pro}\mlabel{pro:2.21}
 Let $(A, M, R, S, \circ)$ be a curved $\mathcal{O}$-operator system associated to $(M, \triangleright, \triangleleft)$. We  define, for all $x, y\in M$,
 \begin{equation}\mlabel{eq:2.27-a}
 x\diamond y=R(x)\triangleright y+x\triangleleft S(y).
 \end{equation}
 Then $(M, \diamond)$ is an associative algebra if and only if
 \begin{equation}\mlabel{eq:2.28-a}
 R(x\circ y)\triangleright z=x\triangleleft S(y\circ z).
 \end{equation}
 \end{pro}

 \begin{proof} Similar to Theorem \mref{thm:2.14}.   \end{proof}

\begin{cor}\mlabel{cor:2.23}
 Let $A$ be an associative algebra, $(M, \circ, \triangleright, \triangleleft)$ be an $A$-bimodule algebra and $R: M\longrightarrow A$ be an $\mathcal{O}$-operator of weight $\lambda$ associated to $(M, \circ, \triangleright, \triangleleft)$. Then the operation given by
 \begin{equation}\mlabel{eq:2.29-a}
 x\odot y=R(x)\triangleright y+x\triangleleft R(y).
 \end{equation}
 is associative if and only if
 \begin{equation}\mlabel{eq:2.30-a}
 \lambda R(x\circ y)\triangleright z=\lambda x\triangleleft R(y\circ z),
 \end{equation}
 where $x, y, z\in M$.
\end{cor}

 \begin{proof} Let $R=S$ and $\omega(x\otimes y)=\lambda (x\circ y)$ in Proposition \mref{pro:2.21}.   \end{proof}

\begin{rmk}\mlabel{rmk:2.24}
  \eqref{eq:2.30-a} is exactly one of the two conditions in the definition of balanced homomorphism (see \cite[ (21)]{BGN1}).
\end{rmk}

\subsection{(Tri)Dendriform systems}

 The notion of dendriform algebra was introduced by J. L. Loday in \cite{Lo}. We extend this notion to dendriform system as follows:

 \begin{defi}\mlabel{de:2.6-a} A {\bf dendriform system $(A, M, \prec, \succ, \lessdot, \gtrdot)$} consists of two linear space $A$, $M$ and four linear maps $\prec: A\otimes A\longrightarrow A$, $\succ: A\otimes A\longrightarrow A$, $\lessdot: M\otimes A\longrightarrow M$ and $\gtrdot: A\otimes M\longrightarrow M$ such that $(A, \prec, \succ)$ is a dendriform algebra, i.e.,
 for all $a, b, c\in A$,
 \begin{equation}\mlabel{eq:2.8-a}
 (a\prec b)\prec c=a\prec (b\succ c+b\prec c),
 \end{equation}
 \begin{equation}\mlabel{eq:2.9-a}
 a\succ (b\prec c)=(a\succ b)\prec c,
 \end{equation}
 \begin{equation}\mlabel{eq:2.10-a}
 a\succ (b\succ c)=(a\succ b+a\prec b)\succ c
 \end{equation}
 hold and for all $x\in M$, the following conditions are satisfied:
 \begin{equation}\mlabel{eq:2.5-a}
 (x\lessdot a)\lessdot b=x\lessdot (a\succ b+a\prec b),
 \end{equation}
 \begin{equation}\mlabel{eq:2.6-a}
 a\gtrdot (x\lessdot b)=(a\gtrdot x)\lessdot b,
 \end{equation}
 and
 \begin{equation}\mlabel{eq:2.7-a}
 a\gtrdot (b\gtrdot x)=(a\succ b+a\prec b)\gtrdot x.
 \end{equation}
 \end{defi}

 \begin{rmk}\mlabel{rmk:ex:3.5}
 \begin{enumerate}[label=\textup{(\alph*)},ref=\textup{\alph*},leftmargin=*]
 \item \mlabel{it:3.5-2} If $A=M$ and $\rc=\dr, \lc=\dl$, then a dendriform system $(A, A, \rc, \lc, \rc, \lc)$ is exactly a dendriform algebra $(A, \rc, \lc)$.
 \item \mlabel{it:3.5-1} Let $(A, \rc, \lc)$ be a dendriform algebra. Define $a\star b=a\rc b+a\lc b$, then $(A, \star)$ is an associative algebra (see \cite{Lo}). Thus if $(A, M, \rc, \lc, \dr, \dl)$ is a dendriform system, then $(M, \dr, \dl)$ is an $(A, \star)$-bimodule by   \eqref{eq:2.5-a}-\eqref{eq:2.7-a}.
 \item \mlabel{it:3.5-3} Let $(A, M, \rc, \lc, \dr, \dl)$ be a dendriform system. Define a linear map $\ast: M\o M\lr M$ by $x\ast y=x\dl y-y\dr x$. Then $(M, \ast)$ is a pre-Lie algebra. 
 \item \mlabel{it:3.5-4} If $\rc=0$, then $(A, \lc)$ is an associate algebra, and  the dendriform system $(A, M, 0, \lc, \dr, \dl)$ is equivalent to $(M, \dr, \dl)$ is an $(A, \lc)$-bimodule. Similarly, if $\lc=0$, then $(A, \rc)$ is an associative algebra, and  the dendriform system $(A, M, \rc, 0, \dr, \dl)$ is equivalent to  $(M, \dr, \dl)$ is an $(A, \rc)$-bimodule.
 \end{enumerate}
 \end{rmk}

\begin{thm}\mlabel{thm:2.7} Let $(A, M, \circ, \triangleright, \triangleleft, R, S)$ be an extended $A$-bimodule algebra. Define linear maps $\succ: M\otimes A\longrightarrow A$, $\prec: A\otimes M\longrightarrow A$ and $\gtrdot, \lessdot: M\otimes M\longrightarrow M$ as follows:
 \begin{equation}\mlabel{eq:2.11-a}
 x\succ a=R(x)a,~a\prec x=a S(x),~x\gtrdot y=R(x)\triangleright y,~x\lessdot y=x\triangleleft S(y)+x\circ y,
 \end{equation}
 where $x, y\in M$ and $a\in A$. If $(A, M, R, S, \circ)$ is a curved $\mathcal{O}$-operator system associated to $(M, \triangleright, \triangleleft)$, then $(M, A, \lessdot, \gtrdot, \prec, \succ)$ is a dendriform system.
 \end{thm}

 \begin{proof} For all $x, y, z\in M$, we can check  \eqref{eq:2.8-a}-\eqref{eq:2.10-a} for $(M, \dr, \dl)$ as follows:
 \begin{eqnarray*}
 (x\dr y)\dr z
 &=&(x\tr S(y))\tr S(z)+(x\tr S(y))\circ z\\
 &=&x\tr (S(y)S(z))+(x\tr S(y))\circ z\\
 &\stackrel{(\mref{eq:2.2-1})}{=}&x\tr S(R(y)\triangleright z+y\triangleleft S(z)+y\circ z)+(x\tr S(y))\circ z\\
 &\stackrel{(\mref{eq:2.3-1})}{=}&x\tr S(R(y)\triangleright z)+x\circ (R(y)\triangleright z)+x\tr S(y\triangleleft S(z))\\
 &&+x\circ (y\triangleleft S(z))+x\tr S(y\circ z)+x\circ(y\circ z)\\
 &=&x\dr (R(y)\triangleright z+y\triangleleft S(z)+y\circ z)=x\dr (x\gtrdot y+x\lessdot y),
 \end{eqnarray*}
 so  \eqref{eq:2.8-a} holds for $(M, \dr, \dl)$. Similarly by  \eqref{eq:2.1-1} and \eqref{eq:2.3-1}, one can prove  \eqref{eq:2.9-a} and \eqref{eq:2.10-a} for $(M, \dr, \dl)$. Then $(M, \dr, \dl)$ is a dendriform algebra. Conditions \eqref{eq:2.5-a}-\eqref{eq:2.7-a} can be checked by \eqref{eq:2.1-1}, \eqref{eq:2.2-1} and the associativity of $A$. Thus $(M, A, \lessdot, \gtrdot, \prec, \succ)$ is a dendriform system. \end{proof}

\begin{rmk}\mlabel{rmk:2.8}
By the proof of Theorem \mref{thm:2.7} and Remark \mref{rmk:2.5-1}, we have: Let $A$ be an associative algebra, $(M, \circ, \triangleright, \triangleleft)$ be an $A$-bimodule algebra and  $(A, M, R, S, \circ)$ be a curved $\mathcal{O}$-operator system associated to $(M, \triangleright, \triangleleft)$. Define linear maps $\succ: M\otimes A\longrightarrow A$, $\prec: A\otimes M\longrightarrow A$, $\gtrdot: M\otimes M\longrightarrow M$ and $\lessdot: M\otimes M\longrightarrow M$ by  \eqref{eq:2.11-a}. Then $(M, A, \lessdot, \gtrdot, \prec, \succ)$ is a dendriform system if and only if  \eqref{eq:2.4-a} holds.
\end{rmk}

\begin{cor}\mlabel{cor:2.9} Let $A$ be an associative algebra, $(M, \circ, \triangleright, \triangleleft)$ be an $A$-bimodule algebra,  $R: M\longrightarrow A$ be a linear map and $\lambda\in K$. Define linear maps $\succ: M\otimes A\longrightarrow A$, $\prec: A\otimes M\longrightarrow A$, $\gtrdot: M\otimes M\longrightarrow M$ and $\lessdot: M\otimes M\longrightarrow M$ as follows:
 $$
 x\succ a=R(x)a,~a\prec x=aR(x),~x\gtrdot y=R(x)\triangleright y,~x\lessdot y=x\triangleleft R(y)+\lambda x\circ y,
 $$
 where $x, y\in M$ and $a\in A$. If $R$ is an $\mathcal{O}$-operator of weight $\lambda$ associated to $(M, \triangleright, \triangleleft)$, then $(M, A, \lessdot, \gtrdot, \prec, \succ)$ is a dendriform system.
\end{cor}

\begin{proof} By Theorem \mref{thm:2.7} and Example \mref{ex:3.2} (\mref{it:2.2-1-4}), we  achieve the proof.\end{proof}

 For $\mathcal{O}$-operator systems, we have the following result.

 \begin{pro}\mlabel{pro:thm:3.6}Let $A$ be an associative algebra, $(M, \tl, \tr)$ be an $A$-bimodule and $R, S: M\lr A$ be two linear maps. Define linear maps, for all $a\in A$ and $x, y\in M$, by
 \begin{eqnarray*}\mlabel{eq:3.13}
 \lc: M\o A\lr A,~x\o a\mapsto R(x)a, && \rc: A\o M\lr A,~~~a\o x\mapsto a S(x),
\\
\mlabel{eq:3.14}
 \dl: M\o M\lr M,~x\o y\mapsto R(x)\tl y, && \dr: M\o M\lr M,~x\o y\mapsto x\tr S(y)
 \end{eqnarray*}
Then,
 \begin{enumerate}[label=\textup{(\alph*)},ref=\textup{\alph*},leftmargin=*]
 \item \mlabel{it:pro:3.5-1}  If $(A, M, R, S)$ is an $\mathcal{O}$-operator system associated to $(M, \tl, \tr)$, then $(M, A, \lessdot, \gtrdot, \prec, \succ)$ is a dendriform system.
 \item \mlabel{it:pro:3.5-2} If $A$ is a non-degenerate associative algebra or $M$ is a faithful $A$-bimodule, and $(M, A, \lessdot, \gtrdot, \prec, \succ)$ is a dendriform system, then $(A, M, R, S)$ is an $\mathcal{O}$-operator system associated to $(M, \tl, \tr)$.
 \end{enumerate}
 \end{pro}

 \begin{proof} (\mref{it:pro:3.5-1}) Let $\circ=0$ in Theorem \mref{thm:2.7}.

\noindent (\mref{it:pro:3.5-2}) By the non-degeneration of $A$ and the proof of  \eqref{eq:2.5-a}-\eqref{eq:2.7-a} or by the faithfulness of $M$ and the proof of \eqref{eq:2.8-a}-\eqref{eq:2.10-a} in Theorem \mref{thm:2.7}, we achieve the proof of part (\mref{it:pro:3.5-2}). \end{proof}

By the above propositions, we  get  relations between $\mathcal{O}$-operator systems and associative algebras and  pre-Lie algebras.

\begin{cor}\mlabel{cor:3.9}  Let $(A, M, R, S)$ be an $\mathcal{O}$-operator system associated to $(M, \tl, \tr)$. Then $(M, \star')$ with $\star': M\o M\lr M$, defined for all
 $x, y\in M$ by
 \begin{equation}\mlabel{eq:3.17}
 x\star' y=R(x)\tl y+x\tr S(y),
 \end{equation}
 is an associative algebra.

 Furthermore, with the maps defined by
 \begin{equation}\mlabel{eq:3.18}
 \dd_\ell: M\o A\lr A,\ x\dd_\ell a=R(x)a, \quad \dd_r: A\o M\lr A,\ a\dd_r x=aS(x),
 \end{equation}
 $(A, \dd_\ell, \dd_r)$ is a $(M, \star')$-bimodule.

 Moreover, in this case, we have that $R$ and $S$ are algebra maps from $(M, \star')$ to $A$, i.e.,
 $$
 R(x)R(y)=R(x\star' y),\quad S(x)S(y)=S(x\star' y).
 $$
 \end{cor}

\begin{proof} It is a consequence of item (\mref{it:pro:3.5-1}) in Proposition \mref{pro:thm:3.6} and item (\mref{it:3.5-1}) in Remark \mref{rmk:ex:3.5}. \end{proof}

\begin{cor}\mlabel{cor:3.10}   Let $(A, M, R, S)$ be an $\mathcal{O}$-operator system associated to $(M, \tl, \tr)$. Then, $(M, \ast)$ with $\ast: M\o M\lr M$, defined for all
 $x, y\in M$ by
 \begin{equation}\mlabel{eq:3.19}
 x\ast y=R(x)\rhd y-y\lhd S(x),
 \end{equation}
 is a pre-Lie algebra.
 \end{cor}

 \begin{proof} By item (\mref{it:pro:3.5-1}) in Proposition \mref{pro:thm:3.6} and item (\mref{it:3.5-3}) in Remark \mref{rmk:ex:3.5}, we get the result.  \end{proof}

 The following notion is a generalization of tridendriform algebra introduced in \cite{LR}.

 \begin{defi}\mlabel{de:2.10} A {\bf tridendriform system} is a septuple $(A, M, \prec, \succ, \cdot, \lessdot, \gtrdot)$ consisting of two linear spaces $A$ and $M$ and five linear maps $\prec: A\otimes M\longrightarrow A$, $\succ: M\otimes A\longrightarrow A$ and  $\lessdot, \gtrdot, \cdot: M\otimes M\longrightarrow M$ such that $(A, \prec, \succ, \cdot)$ is a tridendriform algebra, i.e., for all $a, b, c\in A$,
 \begin{equation}\mlabel{eq:2.15-a}
 (a\rc b)\rc c=a\rc (b\rc c+b\lc c+b\cdot c),
 \end{equation}
 \begin{equation}\mlabel{eq:2.16-a}
 a\lc (b\rc c)=(a\lc b)\rc c,
 \end{equation}
 \begin{equation}\mlabel{eq:2.17-a}
 a\lc (b\lc c)=(a\rc b+a\lc b+a\cdot b)\lc c,
 \end{equation}
 \begin{equation}\mlabel{eq:2.18-a}
 (a\rc b)\cdot c=a\cdot (b\lc c),~(a\lc b)\cdot c=a\lc (b\cdot c),
 \end{equation}
 \begin{equation}\mlabel{eq:2.19-a}
 (a\cdot b)\rc c=a\cdot (b\rc c), (a\cdot b)\cdot c=a\cdot (b\cdot c)
 \end{equation}
 the following conditions are satisfied for all $x\in M$,
 \begin{equation}\mlabel{eq:2.12-a}
 (x\dr a)\dr b=x\dr (a\rc b+a\lc b+a\cdot b),
 \end{equation}
 \begin{equation}\mlabel{eq:2.13-x}
 a\dl (x\dr b)=(a\dl x)\dr b,
 \end{equation}
 \begin{equation}\mlabel{eq:2.14-x}
 a\dl (b\dl x)=(a\rc b+a\lc b+a\cdot b)\dl x.
 \end{equation}
 \end{defi}

 \begin{rmk}\mlabel{rmk:ex:3.12}
 \begin{enumerate}[label=\textup{(\alph*)},ref=\textup{\alph*},leftmargin=*]
 \item \mlabel{it:3.12-2} If $A=M$ and $\rc=\dr, \lc=\dl$, then a tridendriform system $(A, A, \cdot, \rc, \lc, \rc, \lc)$ is exactly a tridendriform algebra $(A, \rc, \lc, \cdot)$.
 \item \mlabel{it:3.12-1} Let $(A, \rc, \lc, \cdot)$ be a dendriform algebra. Define $a\star' b=a\rc b+a\lc b+a\cdot b$, then $(A, \star')$ is an associative algebra (see \cite{Lo}). Thus if $(A, M, \cdot, \rc, \lc, \dr, \dl)$ is a tridendriform system, then $(M, \dr, \dl)$ is an $(A, \star')$-bimodule by \eqref{eq:2.12-a}-\eqref{eq:2.14-x}.
 \item \mlabel{it:3.12-3} When $\cdot=0$, then a tridendriform system $(A, M, \rc, \lc, 0, \dr, \dl)$ is a dendriform system.
 \end{enumerate}
 \end{rmk}

 A dendriform system can be derived from a tridendriform system.

\begin{lem}\mlabel{lem:tridtod} Let $(A, M, \rc, \lc, \cdot, \dr, \dl)$ be a tridendriform system. Define a linear map $\prec':A\o A\lr A$ by
 \begin{eqnarray}
 a\prec' b=a\prec b+a\cdot b,~~\forall~a, b\in A.
 \end{eqnarray}
 Then $(A, M, \rc', \lc, \dr, \dl)$ is a dendriform system.
\end{lem}

\begin{proof} The proof is direct, by checking the conditions in the definition of dendriform systems.
\end{proof}

\begin{thm}\mlabel{thm:2.11} Let $A$ be an associative algebra, $(M, \circ, \triangleright, \triangleleft, R, S)$ be an extended $A$-bimodule algebra. Define linear maps $\succ: M\otimes A\longrightarrow A$, $\prec: A\otimes M\longrightarrow A$ and $\gtrdot, \lessdot, \cdot: M\otimes M\longrightarrow M$ as follows:
 \begin{equation}\mlabel{eq:2.20-a}
 x\succ a=R(x)a,~a\prec x=aS(x),~x\gtrdot y=R(x)\triangleright y,~x\lessdot y=x\triangleleft S(y), ~x\cdot y=x\circ y,
 \end{equation}
 where $x, y\in M$ and $a\in A$. If $(A, M, R, S, \circ)$ is a curved $\mathcal{O}$-operator system associated to $(M, \triangleright, \triangleleft)$, then $(M, A, \lessdot, \gtrdot, \cdot, \prec, \succ)$ is a tridendriform system.
\end{thm}
\begin{proof} The proof is similar to Theorem \mref{thm:2.7}. \end{proof}

\begin{cor}\mlabel{cor:2.13}
 Under the assumption of Corollary \mref{cor:2.9}, let $\succ: M\otimes A\longrightarrow A$, $\prec: A\otimes M\longrightarrow A$ and $\gtrdot, \lessdot, \cdot: M\otimes M\longrightarrow M$ be linear maps defined as follows, for $x, y\in M$ and $a\in A$,
 $$
 x\succ a=R(x)a,\quad a\prec x=aR(x),\quad x\gtrdot y=R(x)\triangleright y,\quad x\lessdot y=x\triangleleft R(y),\quad x\cdot y=\lambda x\circ y.
 $$
 If $R$ is an $\mathcal{O}$-operator of weight $\lambda$ associated to $(M, \triangleright, \triangleleft)$, then $(M, A, \lessdot, \gtrdot, \cdot, \prec, \succ)$ is a tridendriform system.
\end{cor}

 \begin{proof} By Theorem \mref{thm:2.11} and Example \mref{ex:3.2} (\mref{it:2.2-1-4}), we finish the proof. \end{proof}

 \section{$\mathcal{O}$-operator systems and associative Yang-Baxter pairs}  \hskip\parindent
 The following theorem establishes a close relationship between $\mathcal{O}$-operator systems and associative Yang-Baxter pairs introduced in \cite{Br1}. Now let us recall some basic facts on the dual space from \cite{Bai1}.

 For a linear space $M$, denote the usual pairing between the dual space $M^*$ and $M$ by
 $$
 \langle \, ,\rangle : M^*\ti M\lr K, ~\langle  f, x \rangle =f(x), ~\forall~ x\in M, f\in M^*.
 $$

 Let $M$ be a linear space, $M^*$ its dual space and $r=r^1\o r^2\in M\o M$. Define two linear maps as follows:
 \begin{equation}\mlabel{eq:2.6}
 r, r^t: M^*\lr M,~r(f)=r^1\langle  f, r^2 \rangle   ,~~~r^t(f)=\langle  f, r^1 \rangle   r^2, ~~\forall~f\in M^*.
 \end{equation}
 If $r=-r^t$, then we call $r$ {\bf skew-symmetric}.

 Let $(A,\c)$ be an associative algebra and $(M, \tl, \tr)$ be an $A$-bimodule. Define
 \begin{eqnarray}
 &\langle  a\blacktriangleright f, x \rangle   =\langle  f, x\tr a \rangle,\qquad
  \langle  f\blacktriangleleft a, x \rangle   =\langle  f, a\tl x \rangle,&\mlabel{eq:2.6-1}
 \end{eqnarray}
 where $a\in A, f\in M^*$ and $x\in M$. Then $(M^*, \btl, \btr)$ is an $A$-bimodule, which we call the {\bf dual bimodule of $(M, \tl, \tr)$}.

 Specially, we denote the dual bimodule of $(A, \c_\ell, \c_r)$ by $(A^*, \bu_\ell, \bu_r)$, where $\c_\ell, \c_r$ denote the left and right multiplication, respectively.

 \begin{thm}\mlabel{thm:4.1} Let $(A, \c)$ be an associative algebra and $r, s\in A\o A$, which are identified as linear maps from $A^*$ to $A$.
 \begin{enumerate}[label=\textup{(\alph*)},ref=\textup{\alph*},leftmargin=*]
 \item \mlabel{it:thm:4.1-1} $(r, s)$ is an associative Yang-Baxter pair if and only if $r$ and $s$ satisfy, for $a^*, b^*\in A^*$,
 \begin{eqnarray}\mlabel{eq:4.1}
 r(a^*)r(b^*)&=& r(r(a^*)\bu_\ell b^*)-r(a^*\bu_r s^t(b^*)),
 \\
 \mlabel{eq:4.2}
 s^t(a^*)s^t(b^*) &=& -s^t(r(a^*)\bu_\ell b^*)+s^t(a^*\bu_r s^t(b^*)).
 \end{eqnarray}
 \item \mlabel{it:thm:4.1-2} If $r$ and $s$ are skew-symmetric, then $(r, s)$ is an associative Yang-Baxter pair if and only if $(A, A^*, r, s)$ is an $\mathcal{O}$-operator system associated to $(A^*, \bu_\ell, \bu_r)$.
 \end{enumerate}
 \end{thm}

 \begin{proof} It is obvious that $(A, \c_\ell, \c_r)$ is an $A$-bimodule, where $\c_\ell$ and $\c_r$ denote the left and right multiplications, respectively. Then $(A^*, \bu_\ell, \bu_r)$ is the dual bimodule of $(A, \c_\ell, \c_r)$, and for all $a,b\in A, a^*\in A^*$,
 \begin{equation}\mlabel{eq:4.3}
 \langle  a\bu_\ell a^*, b \rangle   =\langle  a^*, b\c_r a \rangle,\quad \langle  a^*\bu_r a, b \rangle   =\langle  a^*, a\c_\ell b \rangle.
 \end{equation}

 \noindent (\mref{it:thm:4.1-1}) For all $a^*, b^*, c^*\in A^*$, by  \eqref{eq:2.6} and \eqref{eq:4.3}, we have
 \begin{eqnarray*}
 \langle  a^*, r^1\c \bar{r}^1 \rangle   \langle  b^*, \bar{r}^2 \rangle   \langle  c^*, r^2 \rangle &=& \langle  a^*, r(c^*)\c r(b^*) \rangle,\\
 -\langle  a^*, r^1 \rangle   \langle  b^*, r^2\c \bar{r}^1 \rangle   \langle  c^*, \bar{r}^2 \rangle
 &=&-\langle  a^*, r^1 \rangle   \langle  b^*, r^2\c r(c^*) \rangle   \\
 &=&-\langle  a^*, r^1 \rangle   \langle  r(c^*)\bu_\ell b^*, r^2 \rangle   \\
 &=&-\langle  a^*, r(r(c^*)\bu_\ell b^*) \rangle, \\*[0,2cm]
 \langle  a^*, r^1 \rangle   \langle  b^*, s^1 \rangle   \langle  c^*, s^2\c r^2 \rangle
 &=&\langle  a^*, r^1 \rangle   \langle  c^*, s^t(b^*)\c r^2 \rangle   \\
 &=&\langle  a^*, r^1 \rangle   \langle  c^*\bu_r s^t(b^*), r^2 \rangle   \\
 &=&\langle  a^*, r(c^*\bu_r s^t(b^*)) \rangle   .
 \end{eqnarray*}
 Then  \eqref{eq:2.4} $\Leftrightarrow$ \eqref{eq:4.1}.

 Similarly, we have
 \begin{eqnarray*}
 \langle  a^*, s^1 \rangle   \langle  b^*, \bar{s}^1 \rangle   \langle  c^*, \bar{s}^2s^2 \rangle   &=& \langle  c^*, s^t(b^*)\c s^t(a^*) \rangle,
 \\
 -\langle  a^*, s^1 \rangle   \langle  b^*, s^2\c \bar{s}^1 \rangle   \langle  c^*, \bar{r}^2 \rangle
 &=&-\langle  b^*, s^t(a^*)\c \bar{s}^1 \rangle   \langle  c^*, \bar{r}^2 \rangle   \\
 &=&-\langle  b^*\bu_r s^t(a^*), \bar{s}^1 \rangle   \langle  c^*, \bar{r}^2 \rangle   \\
 &=&-\langle  c^*, s^t(b^*\bu_r s^t(a^*)) \rangle, \\*[0,2cm]
 \langle  a^*, s^1\c r^1 \rangle   \langle  b^*, r^2 \rangle   \langle  c^*, s^2 \rangle
 &=&\langle  a^*, s^1\c r(b^*) \rangle   \langle  c^*, s^2 \rangle   \\
 &=&\langle  r(b^*)\bu_\ell a^*, s^1 \rangle   \langle  c^*, s^2 \rangle   \\
 &=&\langle  c^*, s^t(r(b^*)\bu_\ell a^*) \rangle   .
 \end{eqnarray*}
 Then  \eqref{eq:2.5} $\Leftrightarrow$  \eqref{eq:4.2}. Therefore, we finish the proof of item (\mref{it:thm:4.1-1}).

\noindent (\mref{it:thm:4.1-2}) If $r$ and $s$ are skew-symmetric, then  \eqref{eq:4.1} and \eqref{eq:4.2} are exactly  \eqref{eq:2.1-1} and \eqref{eq:2.2-1} (when $\om=0$) for $r, s$, respectively. And the rest of the proof of item (\mref{it:thm:4.1-2}) can be obtained by item (\mref{it:thm:4.1-1}). \end{proof}

\begin{cor}\mlabel{cor:4.2} Let $(A, \c)$ be an associative algebra and $r\in A\o A$, which is identified as a linear map from $A^*$ to $A$. If $r$ is skew-symmetric, then
 \begin{enumerate}[label=\textup{(\alph*)},ref=\textup{\alph*},leftmargin=*]
 \item \mlabel{it:cor:4.2-1}  $r$ is a solution of the associative Yang-Baxter equation in $A$ if and only if $r$ is an $\mathcal{O}$-operator associated to $(A^*, \bu_\ell, \bu_r)$.
 \item \mlabel{it:cor:4.2-2} $r$ is a solution of the Frobenius-separability or FS-equation (see \textup{\cite{CMZ}})
 $$
 r_{12}r_{23}=r_{23}r_{13}=r_{13}r_{12}
 $$
 if and only if $(A, A^*, r, 0)$ and $(A, A^*, 0, r)$ are $\mathcal{O}$-operator systems associated to $(A^*, \bu_\ell, \bu_r)$.
 \end{enumerate}
 \end{cor}

 \begin{proof} (\mref{it:cor:4.2-1}) When $r=s$ in Definition \mref{de:2.4}, then \eqref{eq:2.4} and \eqref{eq:2.5} are all  associative Yang-Baxter equations in \cite{Ag1}. Then we  finish the proof by Theorem \mref{thm:4.1}.

\noindent (\mref{it:cor:4.2-2}) Let $s=0$ and $r=0$ in Theorem \mref{thm:4.1}. \end{proof}

\begin{cor}\mlabel{cor:4.3} Let $(A, \c)$ be an associative algebra and $r, s\in A\o A$, which are identified as linear maps from $A^*$ to $A$. If $r$ and $s$ are skew-symmetric, then $(A, A^*, r, s)$ is an $\mathcal{O}$-operator system associated to $(A^*, \bu_\ell, \bu_r)$ if and only if $(A, \d_r, \d_s, \D)$ is a quasitriangular covariant bialgebra, where for all $a\in A$,
 $$
 \d_r(a)=a\c r^1\o r^2-r^1\o r^2\c a,
 $$
 $$
 \d_s(a)=a\c s^1\o s^2-s^1\o s^2\c a,
 $$
 $$
 \D(a)=a\c r^1\o r^2-s^1\o s^2\c a.
 $$
 \end{cor}
\begin{proof} It is a direct consequence of Theorem \mref{thm:4.1} and \cite[Proposition 3.15 and Corollary 3.17]{Br1}. \end{proof}

\section{A special case}\mlabel{se:sc}
  In this section we consider a special case of curved $\mathcal{O}$-operator systems. First let us recall the notion of totally compatible associative dialgebras introduced in \mcite{ZBG}, which can be seen as the Koszul dual of compatible associative algebra.

 \begin{defi}(\cite[Definition 2.1]{ZBG})\mlabel{de:3.1-a} Let $A$ be a linear space and $\mu, \nu: A\otimes A\longrightarrow A$ be two linear maps such that
 $$
 \mu(\mu\otimes \id)=\mu(\id\otimes \mu),~\nu(\nu\otimes \id)=\nu(\id\otimes \nu),
 $$
 $$
 \nu(\mu\otimes \id)=\mu(\id\otimes \nu)=\mu(\nu\otimes \id)=\nu(\id\otimes \mu).
 $$
 Then we call $(A, \mu, \nu)$ a {\bf totally compatible associative dialgebra} or $(\mu, \nu)$ an {\bf associative compatible pair on $A$}.
\end{defi}

\begin{rmk}\mlabel{rmk:3.2-a}
 \begin{enumerate}[label=\textup{(\alph*)},ref=\textup{\alph*},leftmargin=*]
 \item \mlabel{it:rmk:3.2-a1} If $(A, \mu)$ is an associative algebra, then $(\mu, \mu)$ is an associative compatible pair on $A$.
 \item \mlabel{it:rmk:3.2-a2} If for all $x, y\in A$, we write $\mu(x\otimes y)=xy$ and $\nu(x\otimes y)=x\diamond y$, then the
 above compatibility conditions can be rewritten for $x, y, z\in A$ as
 \begin{equation}\mlabel{eq:3.1-a}
 (x y)z=x(y z),~~(x\diamond y)\diamond z=x\diamond (y\diamond z),
 \end{equation}
 \begin{equation}\mlabel{eq:3.2-a}
 (x y)\diamond z=x(y\diamond z)=(x\diamond y)z=x\diamond (y z).
 \end{equation}
 \item \mlabel{it:rmk:3.2-a3} $(\mu, \nu)$ is an associative compatible pair on $A$ if and only if $(A, \nu, \mu_\ell, \mu_r)$ is an $(A, \mu)$-bimodule algebra.
 \end{enumerate}
\end{rmk}

\begin{defi}\mlabel{de:3.3-a} Let $A$ be a vector space and $R: A\longrightarrow A$, $\mu, \nu: A\otimes A\longrightarrow A$ (write $\mu(x\otimes y)=xy$ and $\nu(x\otimes y)=x\diamond y$) be three linear maps such that, for $x, y\in A$,
 \begin{equation}\mlabel{eq:3.3-a}
 R(x)R(y)=R(R(x)y+xR(y)+x\diamond y).
 \end{equation}
Then we call $(A, \mu, \nu, R)$ a {\bf generalized Rota-Baxter algebra}.
 \end{defi}

\begin{rmk}\mlabel{rmk:3.4-a}
 \begin{enumerate}[label=\textup{(\alph*)},ref=\textup{\alph*},leftmargin=*]
 \item \mlabel{it:rmk:3.4-a1} If $(A, \mu)$ is an associative algebra and $x\diamond y=\lambda \mu(x\otimes y)$, then $(A, \mu, \nu, R)$ is a Rota-Baxter algebra of weight $\lambda$.
 \item \mlabel{it:rmk:3.4-a2}If $(A, \mu)$ is an associative algebra and $\omega(x\otimes y)=R(x\diamond y)$, then the generalized Rota-Baxter algebra $(A, \mu, \nu, R)$ is exactly Brzezi\'{n}ski's curved Rota-Baxter algebra (when $R=S$ in \cite[Definition 1.1]{Br2}).
 \end{enumerate}
\end{rmk}

\begin{ex}\mlabel{ex:3.4a}
Let $A$ be a 2-dimensional associative algebra
where the multiplication is defined, with respect to a basis  $\{e_1, e_2\}$,  by
\begin{equation*}
 e_1\cdot e_1=e_1, \quad e_1\cdot  e_2=  e_2,
\quad e_2\cdot e_1=e_2, \quad e_2 \cdot e_2=
e_2.
\end{equation*}
It forms an associative compatible pair with the algebra defined with respect to the same basis by
\begin{align*}
& e_1\diamond e_1=(a-b)e_1+ b e_2, \quad e_1\diamond   e_2= a e_2,
\quad e_2\diamond  e_1=a e_2,  \quad e_2 \diamond  e_2= a e_2,
\end{align*}
where $a$ and $b$ are parameters.
The following linear maps $R$ together with the associative compatible pair result in generalized Rota-Baxter algebras:
\begin{enumerate}[label=\textup{\arabic*)},ref=\textup{\arabic*},leftmargin=1cm]
\item $R(e_1)=(b-a)e_1, \  R(e_2)=0,$
\item $R(e_1)=-a e_1, \  R(e_2)=-a e_1,$
\item $R(e_1)=-b e_2, \  R(e_2)=-a e_2,$
\item $R(e_1)=b(e_1-e_2), \  R(e_2)=a(e_1- e_2),$
\item $R(e_1)=(b-a)e_1-be_2, \  R(e_2)=-a e_2,$
\item $R(e_1)=-a e_2, \  R(e_2)=-a  e_2,$
\item $R(e_1)=(b-a)e_1-a e_2, \  R(e_2)=-a e_2,$
\item $R(e_1)=a(e_1-e_2), \  R(e_2)=a( e_1- e_2),$
\item $R(e_1)=(a-b) e_2, \  R(e_2)=0,$
\item $R(e_1)=(b-2 a) e_1+(a-b)e_2, \  R(e_2)=-a e_1,$
\item $R(e_1)=(b-a) e_1+(a-b)e_2, \  R(e_2)=0.$
\end{enumerate}
\end{ex}

 Similar to the results in Section \mref{se:cos}, we get the following statements.

\begin{pro}\mlabel{pro:3.5-a}
 Let $(A, \mu, \nu, R)$ be a generalized Rota-Baxter algebra such that $(\mu, \nu)$ is an associative compatible pair on $A$.
 \begin{enumerate}[label=\textup{(\alph*)},ref=\textup{\alph*},leftmargin=*]
 \item \mlabel{it:pro:3.5-a1}  Set
 $$
 x\prec y=xR(y)+x\diamond y,\quad x\succ y=R(x)y.
 $$
 Then $(A, \prec, \succ)$ is a dendriform algebra.
 \item \mlabel{it:pro:3.5-a2} Set
 $$
 x\circ y=R(x)y-yR(x)+x\diamond y.
 $$
 Then $(A, \circ)$ is a pre-Lie algebra.
 \item \mlabel{it:pro:3.5-a3} Set
 $$
 x\prec y=xR(y), \quad x\succ y=R(x)y, \quad x\cdot y=x\diamond y.
 $$
 Then $(A, \prec, \succ, \cdot)$ is a tridendriform algebra.
 \item \mlabel{it:pro:3.5-a4}  Set
 $$
 x\ast y=R(x)y+xR(y)+x\diamond y.
 $$
 Then $(A, \ast)$ is an associative algebra.
 \end{enumerate}
\end{pro}

\begin{rmk}\mlabel{rmk:3.6-a}
 \begin{enumerate}[label=\textup{(\alph*)},ref=\textup{\alph*},leftmargin=*]
 \item \mlabel{it:rmk:3.6-a1} If we set $x\diamond y=\lambda xy$ in Proposition \mref{pro:3.5-a}, then we  get the cases for the usual Rota-Baxter algebra of weight $\lambda$.
 \item \mlabel{it:rmk:3.6-a2} By Remark \mref{rmk:3.4-a}\mref{it:rmk:3.4-a2}, we can transfer the results in \cite{Br2} for Brzezi\'{n}ski's curved Rota-Baxter algebras to the case of generalized Rota-Baxter algebras. For example, we have: Let $(A, \mu, \nu, R)$ be a generalized Rota-Baxter algebra such that $(A, \mu)$ is an associative algebra. Set
 $$
 x\ast y=R(x)y+xR(y).
 $$
 Then $(A, \ast)$ is an associative algebra if and only if, for all $x, y, z\in A$,
 $$
 xR(y\diamond z)=R(x\diamond y)z.
 $$
 In particular, if $(A, \mu)$ has an identity, then $(A, \ast)$ is an associative algebra if and only if there exists a central element $\kappa \in A$ such that, for all $x, y\in A$,
 $$
 R(x\diamond y)=\kappa xy.
 $$
 \end{enumerate}
\end{rmk}

Following the definition of generalized Rota-Baxter algebra, replacing $R\circ \omega, S\circ \omega$ by $\omega_1, \omega_2$, respectively, we can also consider the double curved Rota-Baxter system as follows.

\begin{defi}\mlabel{de:3.7-a} A system $(A, R, S, \omega_1, \omega_2)$ consisting of an associative (but not necessarily unital) algebra $(A, \mu)$ and four linear maps $R, S: A\longrightarrow A$, $\omega_1, \omega_2: A\otimes A\longrightarrow A$ is called a {\it double curved Rota-Baxter system} if, for all $a, b\in A$,
 \begin{eqnarray}\mlabel{eq:3.4-a}
 R(a)R(b)=R(R(a)b+aS(b))+\omega_1(a\otimes b),
 \\
 \mlabel{eq:3.5-a}
 S(a)S(b)=S(R(a)b+aS(b))+\omega_2(a\otimes b).
 \end{eqnarray}
\end{defi}

 For double curved Rota-Baxter system $(A, R, S, \omega_1, \omega_2)$, we have the following statement.
\begin{cor}\mlabel{cor:3.8-a}
 Let $(A, R, S, \omega_1, \omega_2)$ be a double curved Rota-Baxter system. Define
 \begin{equation}\mlabel{eq:3.6-a}
 a\ast b=R(a)b+aS(b).
 \end{equation}
 Then $(A, \ast)$ is an associative algebra if and only if, for all $a, b, c\in A$,
 \begin{equation}\mlabel{eq:3.7-a}
 \omega_1(a\otimes b)c=a\omega_2(b\otimes c).
 \end{equation}
\end{cor}

 \begin{proof} Let $M=A$ in Proposition \mref{pro:2.21}.\end{proof}

 \begin{ex}\mlabel{ex:3.10-a}
We consider again  a 2-dimensional associative algebra $A$
where the multiplication is defined, with respect to a basis  $\{e_1, e_2\}$,  by
\begin{equation*}
e_1\cdot e_1=e_1, \quad e_1\cdot  e_2=  e_2,
\quad e_2\cdot e_1=e_2,  \quad e_2 \cdot e_2=
e_2.
\end{equation*}
We construct in the following  a family of examples of associated double curved Rota-Baxter systems $(A, R, S, \omega_1, \omega_2)$.
The maps $R, S, \omega_1, \omega_2$ are defined with respect to the basis $\{ e_1,e_2\}$ by
\begin{align*}
 & R(e_1) =-c_{1,1}^2e_2, \quad    R(e_2)=-e_2,\\
 & S(e_1) =0, \quad   S(e_2)=(c_{1,1}^2-c_{1,2}^2)e_2,
\end{align*}
\begin{align*}
&  \omega_1(e_1, e_1)=c_{1,1}^2(1-c_{1,1}^2)e_2, && \omega_1(e_1, e_2)= (c_{1,2}^2-c_{1,1}^2)e_2 ,\\
& \omega_1(e_2, e_1)=b_{2,1}^1e_1+(1-c_{1,1}^2-b_{2,1}^1c_{1,1}^2)e_2,  && \omega_1(e_2 , e_2)=
b_{2,2}^1e_1+(c_{1,2}^2-c_{1,1}^2-b_{2,2}^1 c_{1,1}^2)e_2,
\end{align*}
\begin{align*}
&  \omega_2(e_1, e_1)=c_{1,1}^1e_1+c_{1,1}^2e_2, && \omega_2(e_1, e_2)=  c_{1,2}^1e_1+c_{1,2}^2e_2,\\
& \omega_2(e_2, e_1)=e_1+e_2,  && \omega_2(e_2 , e_2)=
e_2,
\end{align*}
where $b_{i,j}^k,c_{i,j}^k$ are parameters.
\end{ex}

 \smallskip

 Next we introduce the notion of double curved weak pseudotwistor:

 \begin{defi}\mlabel{de:3.9-a} Let $(A, \mu)$ be an associative algebra. A linear map $T: A\otimes A\longrightarrow A\otimes A$ is called a {\bf double curved weak pseudotwistor} if there exist linear maps $\mathcal{T}: A\otimes A\otimes A\longrightarrow A\otimes A\otimes A$ and $\omega_i: A\otimes A\longrightarrow A, i=1,2$, rendering commutative the following diagrams: \\
\begin{equation}\mlabel{eq:3.8-a}
\begin{minipage}{13cm}
\unitlength 1mm 
\linethickness{0.4pt}
\ifx\plotpoint\undefined\newsavebox{\plotpoint}\fi 
\begin{picture}(121.5,39.5)(0,0)
\put(-7,4.5){$A\otimes A\otimes A$}
\put(-7,36.25){$A\otimes A\otimes A$}
\put(115.5,5.25){$A\otimes A\otimes A$}
\put(115.5,37){$A\otimes A\otimes A$}
\put(12.5,5.25){\vector(1,0){45.25}}
\put(12.5,37){\vector(1,0){45.25}}
\put(114.75,5.5){\vector(-1,0){45.25}}
\put(114.75,37.25){\vector(-1,0){45.25}}
\put(58.25,4){$A\otimes A$}
\put(58.25,35.75){$A\otimes A$}
\put(8,10.25){\vector(0,1){22.75}}
\put(118,10.25){\vector(0,1){22.75}}
\put(62.25,32.5){\vector(0,-1){22.75}}
\put(33.5,39.5){$\id\otimes \mu$}
\put(92.5,39){$\mu\otimes \id$}
\put(65.75,21.25){$T$}
\put(-3.0,19.75){$\id\otimes T$}
\put(121.5,19.5){$T\otimes \id$}
\put(16.5,.75){$(\id\otimes \mu)\circ \mathcal{T}-\id\otimes \omega_2$}
\put(78.25,1.5){$(\mu\otimes \id)\circ \mathcal{T}-\omega_1\otimes \id$}
\end{picture}
\end{minipage}
\end{equation}
\vspace{0,5cm}
\begin{equation}\mlabel{eq:3.9-a}
\begin{minipage}{7cm}
\unitlength 1mm 
\linethickness{0.4pt}
\ifx\plotpoint\undefined\newsavebox{\plotpoint}\fi 
\begin{picture}(62,38.5)(0,0)
\put(4.5,2.25){$A\otimes A$}
\put(17.5,4.25){\vector(1,0){42.25}}
\put(17.5,35.75){\vector(1,0){42.25}}
\put(61.25,2.25){$A$}
\put(15.25,30.5){\vector(0,-1){21.75}}
\put(61.25,30.5){\vector(0,-1){21.75}}
\put(-1,35.25){$A\otimes A\otimes A$}
\put(60.75,35){$A\otimes A$}
\put(62,18.75){$\mu$}
\put(32,2){$\mu$}
\put(32,38.5){$\id\otimes \omega_2$}
\put(0,19.25){$\omega_1\otimes \id$}
\end{picture}
\end{minipage}
\end{equation}

 The map $\mathcal{T}$ is called a {\bf weak companion of $T$} and $\omega_i, i=1,2$ is called the {\bf curvatures of $T$}.
\end{defi}

 In what follows, we consider the relations between the product $(\mref{eq:3.6-a})$ and double curved weak pseudotwistor.
 Firstly, one has the following statement.
\begin{pro}\mlabel{pro:3.10-a}
 Let $T: A\otimes A\longrightarrow A\otimes A$ be a double curved weak pseudotwistor with weak companion $\mathcal{T}$ and curvatures $\omega_i, i=1,2$. Then $\mu\circ T$ is an associative product on $A$.
\end{pro}
 \begin{proof} With the help of associativity of $\mu$ one easily computes,
 \begin{eqnarray*}
 \mu\circ T\circ (\id\otimes \mu\circ T)
 &\stackrel{\eqref{eq:3.8-a}}{=}&\mu\circ (\id\otimes \mu)\circ \mathcal{T}-\mu\circ (\id \circ \omega_2)\\
 &\stackrel{\eqref{eq:3.9-a}}{=}&\mu\circ (\mu \circ \id)\circ \mathcal{T}-\mu\circ (\omega_1\circ \id)\\
 &\stackrel{\eqref{eq:3.8-a}}{=}&\mu\circ T\circ (\mu\circ T \circ \id),
 \end{eqnarray*}
 finishing the proof.\end{proof}
Secondly, we can get:

\begin{pro}\mlabel{pro:3.11-a}
 Let $(A, R, S, \omega_1, \omega_2)$ be a double curved Rota-Baxter system with the curvatures $\omega_i, i=1,2$ that satisfy  \eqref{eq:3.7-a}, and define, for all $a, b, c\in A$,
 \begin{align*}
 & T(a\otimes b)=R(a)\otimes b+a\otimes S(b),
 \\
 & \mathcal{T}(a\otimes b\otimes c)=R(a)\otimes R(b)\otimes c+R(a)\otimes b\otimes S(c)+a\otimes S(b)\otimes S(c).
 \end{align*}
 Then $T$ is a double curved weak pseudotwistor with weak companion $\mathcal{T}$ and curvatures $\omega_i, i=1,2$.
\end{pro}

 \begin{proof} The commutativity of diagram $(\mref{eq:3.9-a})$ is equivalent to \eqref{eq:3.7-a}. To check the commutativity of the left square in diagram $\eqref{eq:3.8-a}$, let us take any $a, b, c\in A$ and compute
 \begin{eqnarray*}
 T\circ (\id\otimes \mu\circ T)(a\otimes b\otimes c)
 &=&T(a\otimes R(b)c+a\otimes bS(c))\\
 &=&R(a)\otimes (R(b)c+bS(c))+a\otimes S(R(b)c+bS(c))\\
 &\stackrel{\eqref{eq:3.5-a}}{=}&R(a)\otimes R(b)c+R(a)\otimes bS(c)+a\otimes S(b)S(c)-a\otimes \omega_2(b\otimes c)\\
 &=&((\id\otimes \mu)\circ \mathcal{T}-\id\circ \omega_2)(a\otimes b\otimes c).
 \end{eqnarray*}
 In addition,
 \begin{eqnarray*}
 T\circ (\mu\circ T\otimes \id)(a\otimes b\otimes c)
 &=&T(R(a)b\otimes c+aS(b)\otimes c)\\
 &=&R(R(a)b+aS(b))\otimes c+(R(a)b+aS(b))\otimes S(c)\\
 &\stackrel{\eqref{eq:3.4-a}}{=}&R(a)R(b)\otimes c+R(a)b\otimes S(c)+aS(b)\otimes S(c)-\omega_1(a\otimes b)\otimes c\\
 &=&((\mu\otimes \id)\circ \mathcal{T}-\omega_1\circ \id)(a\otimes b\otimes c).
 \end{eqnarray*}
 So the commutativity of the right square in diagram $\eqref{eq:3.8-a}$ is checked.  \end{proof}

 For the relation between double curved Rota-Baxter system and pre-Lie algebra, we have:

\begin{cor}\mlabel{cor:3.12-a}
 Let $(A, R, S, \omega_1, \omega_2)$ be a double curved Rota-Baxter system with the curvatures $\omega_i, i=1,2$. Then $A$ with operation $\circ$ defined by
 $$
 a\circ b=R(a)b-bS(a),
 $$
 is a pre-Lie algebra if and only if, for all $a, b\in A$,
 $$
 \omega_1(a\otimes b)-\omega_2(b\otimes a)
 $$
 is in the center of $A$.
\end{cor}

 \section*{Acknowledgment} 
Tianshui Ma is grateful to the Erasmus Mundus project FUSION for supporting the postdoctoral fellowship visiting to M\"alardalen University,
V{\"a}ster{\aa}s,
Sweden and to the Division of Mathematics and Physics,  School of Education, Culture and Communication at M\"alardalen University for cordial hospitality from 2016 to 2017, when the initial research preprint version of the present work was completed \cite{MMS2017arxCOopsyst}.\\
This work is supported by the
Natural Science Foundation of Henan Province (No. 242300421389).

\

 The authors report there are no competing interests to declare.

 \end{document}